\documentclass[12pt,reqno]{amsart}

\usepackage[T1]{fontenc}
\usepackage[utf8]{inputenc}
\usepackage{lmodern}
\usepackage{amsmath,amssymb,amsthm,mathtools}
\usepackage[margin=1.5in]{geometry}
\usepackage{microtype}
\usepackage[hidelinks]{hyperref}

\newtheorem{thm}{Theorem}[section]
\newtheorem{prop}[thm]{Proposition}
\newtheorem{lem}[thm]{Lemma}
\newtheorem{cor}[thm]{Corollary}
\newtheorem{exmpl}[thm]{Example}
\newtheorem{ques}[thm]{Question}

\theoremstyle{definition}
\newtheorem{defn}[thm]{Definition}

\theoremstyle{remark}
\newtheorem{rem}[thm]{Remark}

\newcommand{\Mat}{\operatorname{Mat}}

\newcommand{\rank}{\operatorname{rank}}
\newcommand{\id}{\operatorname{id}}

\title{Amenability is not sensitive to the base field}

\author{Be'eri Greenfeld$^\ast$} \address{Department of Mathematics and Statistics, Hunter College and CUNY Graduate Center, 695 Park Avenue, New York NY 10065, USA} \email{beeri.greenfeld@hunter.cuny.edu}
\thanks{$^\ast$While it is currently prohibited by arXiv policy to list AI as a coauthor, the (human) coauthor is confident that ChatGPT's contribution merits an author credit.}

\date{}

\begin{document}

\begin{abstract}
We prove that the notion of amenability of a (module over) an associative algebra does not depend on the ground field, answering a question proposed by Cornulier in a MathOverflow discussion. A significant part of the argument is based on ideas of ChatGPT 5.6 Sol.
\end{abstract}

\maketitle

\section{Introduction}

Let $K$ be a field and let $A$ be an associative, not necessarily commutative, $K$-algebra. Let $M$ be a (left) $A$-module.

\begin{defn}
Let $a_1,\dots,a_r\in A$ be given. We say that a non-zero finite-dimensional subspace $V\leq M$ is \textbf{$\varepsilon$-almost invariant with respect to $a_1,\dots,a_r$} if
$$\dim_K (V+a_1V+\dots+a_r V) \leq (1+\varepsilon) \dim_K V.$$
We say that $M$ is an \textbf{amenable $A$-module} if for every $a_1,\dots,a_r$ and for every $\varepsilon>0$, $M$ admits an $\varepsilon$-almost invariant subspace with respect to $a_1,\dots,a_r$.
\end{defn}
This notion can be viewed as a linearization of the notion of amenable $G$-sets.

\begin{defn}
An algebra is \textbf{amenable}
 if it is amenable as a module over itself. An algebra is \textbf{fully amenable} if all of its non-zero modules are amenable.
\end{defn}

Bartholdi \cite{Bartholdi_group_algebra} proved that a group is amenable if and only if its group algebra (over an arbitrary field) is amenable. For more on the amenability of algebras, see 
\cite{Ara_BMS,ArzhPau,Atkarskaya,Bartholdi_book,Bartholdi2,BG,CSSV_JMS,CSSV_Expo,DAdderio,Elek_sf,Elek,Gro,Gromov} and references therein.

Notice that the notion of amenability of a module (and thus, of an algebra) is defined relative to the specified base field and can a priori depend on it; it is thus natural to ask:
\begin{ques}
Does the notion of amenability of an algebra depend on the ground field?
\end{ques}
We learned about this question from Y.~Cornulier's comment on MathOverflow, where this question is proposed \cite{YCor} (posed there in the context of amenable algebras). In this note, we settle it.

\begin{thm} \label{thm}
Let $L$ be a field and let $A$ be an $L$-algebra. Let $M$ be an $A$-module. Let $K\subseteq L$ be a subfield. Then $M$ is amenable over $A$ where $A$ is viewed as a $K$-algebra if and only if $M$ is amenable over $A$ where $A$ is viewed as an $L$-algebra.
\end{thm}

An immediate consequence is:

\begin{cor}
Let $L$ be a field and let $A$ be an $L$-algebra. Let $K\subseteq L$ be a subfield. Then $A$ is amenable (respectively, fully amenable) as a $K$-algebra if and only if it is amenable (respectively, fully amenable) as an $L$-algebra.
\end{cor}

The proof of Theorem \ref{thm} is given in Proposition \ref{prop:going up} and Proposition \ref{prop:going down}. We conclude with examples related to exhaustive amenability.

\bigskip

\noindent \textbf{Methodology.} 
This work is heavily based on discussions with OpenAI's ChatGPT 5.6 which provided, in particular, the core ideas behind the results in Section \ref{sec:2}. To the human coauthor, this collaboration resembled a back-and-forth discussion between human collaborators, finding gaps in each other's arguments, sharing ideas regarding special cases, and concluding with a complete argument.
Interestingly, at the beginning of the discussion, ChatGPT claimed that amenability is not preserved under base change, but the discussion evolved.
The paper was written by the human author, who is responsible for it.

\medskip

\noindent \textbf{Acknowledgement.} We thank Laurent Bartholdi for pointing out that the proof can be extended to the amenability of arbitrary modules.

\section{Going up} \label{sec:2}

We start with some conventions and notation. Let $K$ be a field. By $U\leq_K V$ we mean that $U$ is a $K$-vector subspace of $V$ (this will be useful as multiple base fields will be used).

Let $0\neq E,F$ be finite-dimensional $K$-vector spaces. We will be interested in the case where $E\subseteq F$, although the linear-algebraic auxiliary results can be done in a more general setting. Let $\phi_0,\dots,\phi_r\colon E\rightarrow F$ be $K$-linear maps; assume that $\phi_0(v)=v$ for all $v\in E$. For a subspace $U\leq_K E$, denote $$\Phi_K(U) := \sum_{i=0}^{r} \phi_i(U) \leq_K F.$$

Suppose that $K\subseteq M$ is a field extension. For any $T\leq_M M\otimes_K E$, we denote $$\Phi_M(T) := \sum_{i=0}^{r} (\id_M \otimes \phi_i)(T) \leq_M M\otimes_K F.$$

\begin{defn}

Let $K$ be a field, $0\neq E,F$ finite-dimensional $K$-vector spaces, and $$\Phi\colon \{U\leq_K E\} \rightarrow \{V\leq_K F\}$$ any operator assigning a subspace of $F$ to any subspace of $E$. We say that $E$ is \textbf{$\Phi$-optimal over $K$} if $$\frac{\dim_K \Phi(U)}{\dim_K U} \geq \frac{\dim_K \Phi(E)}{\dim_K E}$$ for every $0\neq U\leq_K E$.
\end{defn}

\begin{lem} \label{lem:finite field extensions}
Retain the above notation and further assume that $M/K$ is a finite extension. If $E$ is $\Phi_K$-optimal over $K$, then $M\otimes_K E$ is $\Phi_M$-optimal over $M$.
\end{lem}
\begin{proof}
Suppose that $\dim_K \Phi_K(E) = c\cdot \dim_K E$. Notice that $\Phi_M(M\otimes_K E)=M\otimes_K \Phi_K(E)$, hence 
\begin{eqnarray*}
\dim_M \Phi_M(M\otimes_K E) & = & \dim_M M\otimes_K \Phi_K(E) \\
& = & \dim_K \Phi_K(E) \\ & = & c\cdot \dim_K E \\ & = & c \cdot \dim_M M\otimes_K E.    
\end{eqnarray*}
Fix a $K$-linear basis of $M$, say, $\{b_1,\dots,b_d\}$ with $b_1=1$.
As $K$-vector spaces, $M\otimes_K E\cong E^d$ and $M\otimes_K F\cong F^d$. Under this identification, for each $T\leq_M M\otimes_K E$, the space $\Phi_M(T)$ becomes $\sum_{i=0}^{r} \phi_i^d(T)$ where $\phi_i^d=(\phi_i,\dots,\phi_i)\colon E^d\rightarrow F^d$. Therefore, we must prove that for every $T\leq E^d$, 
\begin{equation} \label{eq:1}
\dim_K \sum_{i=0}^{r} \phi_i^d(T) \geq c \cdot \dim_K T.
\end{equation}
This will show that, pre-identification of $M\otimes_K E\cong E^d,M\otimes_K F\cong F^d$, we have
$$\dim_K \Phi_M(T) \geq c\cdot \dim_K T$$ and thus, dividing both sides by $d=[M\colon K]$,
$$\dim_M \Phi_M(T) \geq c\cdot \dim_M T.$$ This implies that $M\otimes_K E$ is $\Phi_M$-optimal over $M$.

Let us now prove \eqref{eq:1} by induction on $d$. The case $d=1$ is evident: it is equivalent to the assumption that $E$ is $\Phi_K$-optimal over $K$. Suppose that \eqref{eq:1} holds for $d-1$. Let $\pi_E\colon E^d\rightarrow E$ and $\pi_F\colon F^d\rightarrow F$ denote the projections onto the $d$-th component. Let $U := \pi_E(T)\leq_K E$ and $$T_0 := T\cap \ker(\pi_E)\leq_K E^{d-1}\oplus \{0\}.$$ Notice that $T/T_0\cong \pi_E(T)=U$, so 
\begin{equation} \label{eq:5}
\dim_K T = \dim_K T_0 + \dim_K U.
\end{equation}
Now let $Y := \sum_{i=0}^{r} \phi_i^d(T)$. Notice that
\begin{equation} \label{eq:2}
    \sum_{i=0}^{r}\phi_i^d(T_0) \subseteq Y \cap \ker(\pi_F)
\end{equation}
since every vector in $T_0$ takes the form $(e_1,\dots,e_{d-1},0)$ and $\phi_i^d(e_1,\dots,e_{d-1},0)=(\phi_i(e_1),\dots,\phi_i(e_{d-1}),\phi_i(0))$.
Next, we claim that
\begin{equation} \label{eq:3}
    \pi_F(Y) = \Phi_K(U).
\end{equation}
Indeed, every $y\in Y$ takes the form $\sum_{i=0}^{r} \phi_i^d(\vec{t}_i)$ for some $\vec{t}_i=(t_{i1},\dots,t_{id})\in T,\ 0\leq i\leq r$. Thus \begin{eqnarray*} 
\pi_F(y) & = & \sum_{i=0}^{r} \pi_F(\phi_i^d(\vec{t}_i)) \\ & = & \sum_{i=0}^{r} \pi_F(\phi_i(t_{i1}),\dots,\phi_i(t_{id})) \\ & = & \sum_{i=0}^{r} \phi_i(t_{id}) \in \Phi_K(\pi_E(T)) = \Phi_K(U).
\end{eqnarray*}
For the other inclusion, every element $\xi \in \Phi_K(U)$ takes the form $\xi = \sum_{i=0}^{r} \phi_i(t_{id})$ for some $\vec{t}_i=(t_{i1},\dots,t_{id})\in T,\ 0\leq i\leq r$. Now
$$\pi_F \underbrace{\left(\sum_{i=0}^{r} \phi_i^d(\vec{t}_i)\right)}_{\in Y} = \pi_F\left( \sum_{i=0}^{r} \phi_i(t_{i1}),\dots, \sum_{i=0}^{r} \phi_i(t_{id}) \right) = \xi.$$

Now, since $Y/\left(Y\cap \ker(\pi_F) \right)\cong \pi_F(Y)$, we have
\begin{eqnarray} \label{eq:4}
    \dim_K Y & = & \dim_K \left( Y\cap \ker(\pi_F) \right) + \dim_K \pi_F(Y) \nonumber \\ & \overset{\eqref{eq:2},\eqref{eq:3}}{\geq} & \dim_K \left( \sum_{i=0}^{r} \phi_i^d(T_0) \right) + \dim_K \Phi_K (U).
\end{eqnarray}
Recall that $T_0\leq E^{d-1} \oplus \{0\}$, and observe that $$\phi_i^{d}|_{E^{d-1} \oplus \{0\}} = \left(\phi_i^{d-1}|_{E^{d-1}}\right) \oplus 0$$
and therefore, by the induction hypothesis applied to $T_0$,
$$\dim_K \sum_{i=0}^{r} \phi_i^d (T_0) \geq c\cdot \dim_K T_0.$$ Back to \eqref{eq:4},
\begin{eqnarray*}
\dim_K Y & \geq & c\cdot \dim_K T_0 + \dim_K \Phi_K(U) \\ & \geq & c\cdot \dim_K T_0 + c\cdot \dim_K U \\ & \overset{\eqref{eq:5}}{=} & c\cdot \dim_K T,    
\end{eqnarray*}
where the middle inequality follows since $E$ is assumed to be $\Phi_K$-optimal over $K$. The lemma is proved.
\end{proof}

The next step is to lift optimality to arbitrary field extensions.

\begin{lem} \label{lem:arbitrary field extensions}
Retain the above notation and let $M/K$ be an arbitrary field extension. If $E$ is $\Phi_K$-optimal over $K$, then $M\otimes_K E$ is $\Phi_M$-optimal over $M$.
\end{lem}
\begin{proof}
Suppose that $\dim_K \Phi_K(E) = c\cdot \dim_K E$. As in the beginning of Lemma \ref{lem:finite field extensions}, observe that $\dim_M \Phi_M ( M\otimes_K E ) = c \cdot \dim_M M\otimes_K E$. Assume toward a contradiction that there is some $T\leq_M  M\otimes_K E$, say, $\dim_M T = a$, such that $\dim_M \Phi_M(T) < ca$. Denote $Y:=\Phi_M(T)$ and $b:=\dim_M Y$, so $b<ca$.

Our goal will be to construct some finite extension $K\subseteq M_0$ such that $M_0\otimes_K E$ is not $\Phi_{M_0}$-optimal, contradicting Lemma \ref{lem:finite field extensions}. Toward this end, we fix coordinates.

Fix $K$-linear bases $\mathcal{B}_E,\mathcal{B}_F$ for $E,F$, respectively, and view $E\cong K^n,F\cong K^m$ via coordinates in these bases. Subsequently, we obtain $M$-linear isomorphisms $T\leq M\otimes_K E \cong M^n,\ \  Y\leq M\otimes_K F\cong M^m$. We fix $M$-linear bases for $T,Y$ (viewed within $M^n,M^m$, respectively) and place them as column vectors of matrices
$$P\in \Mat_{n\times a}(M),\ \ Q\in \Mat_{m\times b}(M)$$
so the column spaces of $P,Q$ are $T,Y$, respectively, and $\rank(P)=a,\rank(Q)=~b$.
Moreover, for each $0\leq i\leq r$, let $A_i=[\phi_i]_{\mathcal{B}_F}^{\mathcal{B}_E} \in \Mat_{m\times n}(K)$. Thus $A_i T \subseteq Y$, so the column space of $A_i P$ -- which is $A_i T$ -- is contained in the column space of $Q$, which is $Y$. Therefore, there exist matrices $C_i\in \Mat_{b\times a}(M)$ such that 
\begin{equation} \label{eq:6}
A_i P = Q C_i\ \ \text{for all}\ \ 0\leq i\leq r.
\end{equation}
Recall that $\rank(P)=a,\rank(Q)=b$ and therefore there exists a non-zero $a\times a$ minor $\Delta$ in $P$ and a non-zero $b\times b$ minor $\Gamma$ in $Q$. (For simplicity, assume that $b>0$, which is indeed the case in our application, as we will apply this lemma for $\phi_0(v)=v$ so $T\subseteq \Phi_M(T)$. If one is interested in covering the case that $b=0$, take $\Gamma=1$.)

Let $R$ be the $K$-subalgebra of $M$ generated (as an algebra, not as a subfield) by all of the entries of the matrices $P,Q,C_i$, and by $\Gamma^{-1},\Delta^{-1}$. This is an integral domain, affine over $K$. Pick an arbitrary maximal ideal $\mathfrak{m}\triangleleft R$ and let $M_0 := R/\mathfrak{m}$. This is a finite field extension of $K$. We denote by $\bar{\cdot}$ the reduction modulo $\mathfrak{m}$. Notice that by \eqref{eq:6},
\begin{equation} \label{eq:7}
A_i \ \overline{P} = \overline{Q} \ \overline{C}_i
\end{equation}
for all $0 \leq i\leq r$, and $\rank(\overline{P})=a,\rank(\overline{Q})=b$ since these matrices admit invertible $a\times a$ and $b\times b$ minors, respectively.
Let $T_0 \leq M_0\otimes_K E\cong M_0^n$ be the column space of $\overline{P}$, and by the above argument, $\dim_{M_0} T_0 = \rank(\overline{P}) = a$. Considering the column space of the left hand side of \eqref{eq:7}, we see that $(\id_{M_0}\otimes \phi_i)(T_0)$ is contained in the column space of $\overline{Q}$, call it $Y_0$. Since this is true for all $0\leq i\leq r$, we have that
$$\dim_{M_0} \Phi_{M_0}(T_0) \leq \dim_{M_0} Y_0 \leq b < ca = c \cdot \dim_{M_0} T_0,$$
so $M_0\otimes_K E$ is not $\Phi_{M_0}$-optimal over $M_0$, contradicting Lemma \ref{lem:finite field extensions}.
\end{proof}

\begin{prop} \label{prop:going up}
Let $L$ be a field and let $K\subseteq L$ be a subfield. Let $A$ be an $L$-algebra and let $M$ be an $A$-module. If $M$ is amenable where $A$ is a $K$-algebra then $M$ is amenable where $A$ is an $L$-algebra.
\end{prop}

\begin{proof}
Let $A$ be an $L$-algebra and let $M$ be an $A$-module.
Fix a finite subset $a_1,\dots,a_r\in A$ and let $\varepsilon>0$ be given. Put $a_0=1$. By the amenability of $M$ over $A$ as a $K$-algebra, there exists a finite-dimensional subspace $V\leq_K M$ such that 
$$\frac{\dim_K(V+a_1V+\dots+a_rV)}{\dim_K V}\leq 1+\varepsilon.$$ We may further assume that $$\frac{\dim_K(V+a_1V+\dots+a_rV)}{\dim_K V}$$ is minimal possible among all non-zero subspaces of $V$, namely, $V$ is $\Phi_K$-optimal over $K$. Denote $n:=\dim_K V,\ W:=V+a_1V+\dots+a_rV,\ m:=\dim_K W$. Let $\phi_0,\dots,\phi_r\colon V\rightarrow W$ be given by $\phi_i(v)=a_iv,\ 0\leq i\leq r$; notice that, with notations as before, $\Phi_K(V)=W$. 

Consider the $L$-linear maps $\mu_V\colon L\otimes_K V \rightarrow M,\ \mu_W \colon L\otimes_K W\rightarrow M$ given by $\mu_V(\lambda\otimes v)=\lambda v,\ \mu_W(\lambda\otimes w)=\lambda w$ on pure tensors. Let $$P:=\ker(\mu_V) \leq_L L\otimes_K V,\ \ Q:=\ker(\mu_W)\leq_L L\otimes_K W$$ and denote $p:=\dim_L P,\ q := \dim_L Q$. We claim that
\begin{equation} \label{eq:8}
\Phi_L(P) \subseteq Q.
\end{equation}
Indeed, recall that $\Phi_L(P)=\sum_{i=0}^{r} (\id_L\otimes \phi_i)(P)$. Let $\xi\in P,\ \xi = \lambda_1\otimes v_1+\dots+\lambda_t \otimes v_t$ such that $\sum_{i=1}^{t} \lambda_i v_i = 0$. Then for each $0\leq i\leq r$,
\begin{eqnarray*}
(\id_L\otimes \phi_i)(\xi) & = & \lambda_1\otimes \phi_i(v_1)+\dots+\lambda_t \otimes \phi_i(v_t) \\ & = & \lambda_1\otimes a_i v_1+\dots+\lambda_t \otimes a_i v_t
\end{eqnarray*}
and $\mu_W(\lambda_1\otimes a_i v_1+\dots+\lambda_t \otimes a_i v_t)=a_i(\lambda_1 v_1+\dots+\lambda_t v_t)=0$, so $(\id_L\otimes \phi_i)(\xi)\in \ker(\mu_W)=Q$.

Since $V$ is $\Phi_K$-optimal over $K$, by Lemma \ref{lem:arbitrary field extensions}, $L\otimes_K V$ is $\Phi_L$-optimal over $L$; note that $$\frac{\dim_L L\otimes_K W}{\dim_L L\otimes_K V} = \frac{\dim_K W}{\dim_K V} = \frac{m}{n}.$$ Applying this to $P$ (we may assume $P\neq 0$ or else the following is evident), we see that
\begin{equation} \label{eq:9}
\dim_L \Phi_L(P) \geq \frac{m}{n} \dim_L P = \frac{m}{n}p.
\end{equation}
Moreover, using \eqref{eq:8}, we get
\begin{equation} \label{eq:10}
q = \dim_L Q \geq \dim_L \Phi_L(P) \geq \frac{m}{n} p.
\end{equation}
Now observe that 
\begin{eqnarray*}
&& LV \cong  (L\otimes_K V)/P \\
&& LV+a_1LV+\dots+a_rLV \cong (L\otimes_K W)/Q
\end{eqnarray*}
as $L$-vector spaces, so $$\dim_L LV = n - p,\ \ \ \dim_L \left(LV+a_1LV+\dots+a_rLV\right) = m - q.$$ By \eqref{eq:10},
$$m - q \leq  m - \frac{m}{n} p = \frac{m}{n}(n-p)$$
so
$$\frac{\dim_L \left(LV+a_1LV+\dots+a_rLV\right)}{\dim_L LV} = \frac{m - q}{n - p} \leq \frac{m}{n} \leq 1+ \varepsilon.$$
This proves that $LV$ is an $L$-vector subspace of $M$ which is $\varepsilon$-almost invariant with respect to $a_1,\dots,a_r$, so $M$ is amenable over $A$ where $A$ is viewed as an $L$-algebra.
\end{proof}

\section{Going down}

\begin{prop} \label{prop:going down}
Let $A$ be an $L$-algebra and let $K\subseteq L$ be a subfield. Let $M$ be an $A$-module. If $M$ is amenable over $A$ viewed as an $L$-algebra then it is also amenable over $A$ as a $K$-algebra.
\end{prop}
\begin{proof}
Let $a_1,\dots,a_r \in A$ and $\varepsilon>0$ be given. Pick an $n$-dimensional $L$-vector subspace $V=Lv_1+\dots+Lv_n\leq_L M$ such that $$\dim_L \left(V+a_1 V+\dots+a_r V\right) \leq (1+\varepsilon)n.$$
Write
$$V+a_1 V+\dots+a_r V = Lv_1+\dots+Lv_n+L a_{k_1} v_{i_1}+\dots+L a_{k_t} v_{i_t}$$
for some $k_1,\dots,k_t\in \{1,\dots,r\}$ and $i_1,\dots,i_t\in \{1,\dots,n\}$, and $t\leq \varepsilon n$, such that $v_1,\dots,v_n,a_{k_1}v_{i_1},\dots,a_{k_t}v_{i_t}$ form an $L$-linear basis for $V+a_1V+\dots+a_rV$.

For each $1\leq p\leq r$ and $1\leq i\leq n$, we can write
$$a_p v_i = \sum_{j=1}^{n} \alpha_{i,p;j} v_j + \sum_{l=1}^{t} \beta_{i,p;l} a_{k_l} v_{i_l}$$
for suitable $\alpha_{i,p;j},\beta_{i,p;l}\in L$. Let $X$ be the set of all these scalars, and throw in $1$ if it is not one of them. Observe that
\begin{equation} \label{eq:11}
a_p v_i \in KXv_1+\dots+KXv_n+KX a_{k_1} v_{i_1} + \dots + KX a_{k_t} v_{i_t}
\end{equation}
for all $1\leq p\leq r,\ 1\leq i\leq n$.

Since $X$ is a finite subset of a commutative $K$-algebra (namely, $L$), it follows that $$\frac{\dim_K KX^{m+1}}{\dim_K KX^m} \xrightarrow{m\rightarrow \infty} 1.$$ Fix $m\gg 1$ such that $$\dim_K KX^{m+1} \leq (1+\varepsilon) \underbrace{\dim_K KX^{m}}_{=:d} = (1+\varepsilon)d.$$
Let $U := KX^{m} v_1 + \dots + KX^{m} v_n$. Recall that $v_1,\dots,v_n$ are linearly independent over $L$ (and $K,X\subseteq L$), so $\dim_K U = n\cdot d$.
Now, by \eqref{eq:11},
\begin{eqnarray*}
U+a_1U+\dots+a_rU & \subseteq & KX^{m+1}v_1+\dots+KX^{m+1}v_n \\ & & + KX^{m+1} a_{k_1} v_{i_1} + \dots + KX^{m+1} a_{k_t} v_{i_t}
\end{eqnarray*}
and the dimension of the $K$-vector space on the right hand side is at most $$(n+t)\cdot \dim_K KX^{m+1} \leq (n+t)(1+\varepsilon)d.$$ Recall that $t\leq \varepsilon n$, so this dimension is further $\leq (1+\varepsilon)^2 dn$. Finally, recall that $\dim_K U= dn$, so
$$\dim_K\left(U+a_1U+\dots+a_rU\right) \leq (1+\varepsilon)^2 \dim_K U.$$
Since $(1+\varepsilon)^2\xrightarrow{\varepsilon\rightarrow 0} 1$, it follows that for every $\varepsilon>0$ we can find an $\varepsilon$-almost invariant $K$-subspace of $M$ with respect to $a_1,\dots,a_r$, proving that $M$ is amenable over $A$ when viewed as a $K$-algebra. 
\end{proof}

\section{Exhaustive amenability}

\begin{defn}
Let $K$ be a field and let $A$ be a $K$-algebra. Let $M$ be an $A$-module. We say that $M$ is \textbf{exhaustively amenable} if for every $a_1,\dots,a_r\in A$ and for every $\varepsilon>0$, every finite-dimensional subspace $W\leq_K M$ is contained in an $\varepsilon$-almost invariant subspace with respect to $a_1,\dots,a_r$.
\end{defn}

\begin{rem}
If, for every $a_1,\dots,a_r\in A$ and for every $\varepsilon>0$, $M$ admits $\varepsilon$-almost invariant subspaces with respect to $a_1,\dots,a_r$ of arbitrarily large dimensions, then $M$ is exhaustively amenable. Indeed, given any finite-dimensional $V\leq_K M$ of dimension, say, $d:=\dim_K V$, pick an $\varepsilon/2$-almost invariant subspace $W\leq_K M$ with respect to $a_1,\dots,a_r$, with $\dim_K W \geq \frac{2(r+1)d}{\varepsilon}$. 
It follows that $$(1+\varepsilon/2)\dim_K W + (r+1)d \leq (1+\varepsilon) \dim_K W$$
and
\begin{eqnarray*}
\dim_K\left(V+W+\sum_{i=1}^{r} a_i(V+W)\right) & \leq & \dim_K\left(V+\sum_{i=1}^{r} a_i V\right) + \dim_K\left(W+\sum_{i=1}^{r} a_i W\right)  \\ & \leq & (r+1)d + \dim_K(W+\sum_{i=1}^{r} a_i W) \\ & \leq & (r+1)d+(1+\varepsilon/2)\dim_K W \\ & \leq & (1+\varepsilon)\dim_K W \leq (1+\varepsilon) \dim_K (V+W)
\end{eqnarray*}
so $V+W$ is an $\varepsilon$-almost invariant subspace with respect to $a_1,\dots,a_r$, containing $V$.

Conversely, if $M$ is an infinite-dimensional, exhaustively amenable module, then it admits almost invariant subspaces of arbitrarily large dimensions.
\end{rem}

\begin{exmpl} \label{exmpl:free}
Let $A$ be a $K$-algebra with a non-zero finite-dimensional two-sided ideal $I$ such that $A/I$ is non-amenable as a module over itself. Then $A$ is amenable as a module over itself, but not exhaustively amenable.

Indeed, $A$ is amenable as it possesses an invariant subspace $I$; namely, $AI=I$, so $I$ is an $\varepsilon$-almost invariant subspace for every $\varepsilon>0$. However, assume that $A$ is exhaustively amenable and let $a_1,\dots,a_m\in A$ be lifts of arbitrary elements from $A/I$; denote $d:=\dim_K I$. Since $A$ is exhaustively amenable and infinite-dimensional, for every $\varepsilon>0$ and for every $N$, there exists $I\subseteq V\leq_K A$ such that $\dim_K V\geq N$ and $\dim_K\left(V+\sum_{i=1}^{m} a_i V \right)\leq (1+\varepsilon) \cdot \dim_K V$. In particular, we may take $\dim_K V \geq d+d/\varepsilon$. Let $\overline{V}$ denote the image of $V$ modulo $I$. Then $\dim_K \overline{V} = \dim_K V - d \geq d/\varepsilon$, and 
\begin{eqnarray*}
\dim_K\left(\overline{V}+\sum_{i=1}^{m} \bar{a}_i \overline{V}\right) & \leq & (1+\varepsilon) \cdot \dim_K V \\ & = & (1+\varepsilon) \cdot (\dim_K \overline{V} + d)  \\ & \leq & (1+\varepsilon)^2 \cdot \dim_K \overline{V}  
\end{eqnarray*}
and, since $(1+\varepsilon)^2\xrightarrow{\varepsilon\rightarrow 0} 1$, it follows that $A/I$ is amenable, a contradiction.
\end{exmpl}

Unlike amenability, exhaustive amenability is sensitive to the base field: large almost invariant subspaces can `shrink' along an extension of the base field.

\begin{exmpl}[{Exhaustive amenability is sensitive to base change}]
Consider the $K(t)$-algebra $$ A=K(t)\langle x,y,z \rangle / \langle xz,\ zx,\  yz,\ zy,\ z^2 \rangle. $$
Notice that $I:=\langle z \rangle = K(t)z$ is one-dimensional over $K(t)$, and $A/I$ is a non-commutative free $K(t)$-algebra; hence by Example \ref{exmpl:free}, $A$ is not exhaustively amenable as a $K(t)$-algebra.

Now let us view $A$ as a $K$-algebra. Let $\varepsilon>0$ and $a_1,\dots,a_r\in A$ be given. Let $b_1,\dots,b_s\in K(t)$ be rational functions such that the constant terms of $a_1,\dots,a_r$ (those not accompanying non-empty monomials in $x,y,z$) lie in $U:=K+Kt+Kb_1+\dots+Kb_s$. 
Since $U\leq_K K(t)$, a commutative $K$-algebra, there exists some $m_0$ such that for every $m\geq m_0$, we have $\dim_K U^{m+1}\leq (1+\varepsilon)\dim_K U^m$.
Let $V := U^m z \leq_K A$. Notice that $\dim_K V \geq m$. Moreover, for every $1\leq p\leq r$, we have $a_p V = a_p U^m z \subseteq U^{m+1} z$, so $$\dim_K (V+\sum_{i=1}^r a_i V)\leq \dim_K U^{m+1} z \leq (1+\varepsilon) \dim_K U^m = (1+\varepsilon) \dim_K V,$$
so we found almost invariant subspaces of arbitrarily large dimensions.
\end{exmpl}


\begin{thebibliography}{99}
\bibitem{YCor}
Y.~Cornulier, \text{mathoverflow.net/questions/301281/is-there-a-notion-of-amenable-ring?}
\end{thebibliography}
\end{document}